\documentclass[11pt,reqno]{amsart}
\usepackage{amsmath, amsthm, amssymb,xspace}
\usepackage[breaklinks=true]{hyperref}
\usepackage{xcolor}
\usepackage[margin=1.47in]{geometry}

\newtheorem{theorem}{Theorem}
\newtheorem{lemma}[theorem]{Lemma}
\newtheorem{prop}[theorem]{Proposition}
\newtheorem{cor}[theorem]{Corollary}

\theoremstyle{definition}

\newtheorem{remark}[theorem]{Remark}
\newtheorem{exap}[theorem]{Example}
\newtheorem{ques}{Question}

\newtheorem*{proA}{Question A}
\newtheorem*{proB}{Question B}

\begin{document}

\title{The Brown-Halmos theorems on the Fock space}

\author{Jie Qin}
\address{School of Mathematics and Statistics, Chongqing Technology and Business University, 400067, China}
\email{qinjie24520@163.com}

\thanks{The author was supported by the National Natural
Science Foundation of China (11971125, 12071155), the Natural Science Foundation of Chongqing, China (CSTB2022NSCQ-MSX1045).}

\begin{abstract}
In this paper, we extend the Brown-Halmos theorems to the Fock space and investigate the range of the Berezin transform. We observe that there are non-pluriharmonic  functions $u$ that can be written as a finite sum $B(u)=\sum_lf_l\overline{g_l}$, where $f_l,g_l$ are holomorphic functions  belonging to the class $\mathrm{Sym}(\mathbb{C}^n)$. In addition, we solve an open question about the zero product of Toeplitz operators, which was posed by Bauer et al. in 2015. Our results reveal that the Brown-Halmos theorems on the Fock space are more complicated than that on the classical Bergman space.
\end{abstract}

\maketitle

\section{Introduction}\label{s1}

Let $\mathbb{C}^n$ be the complex $n-$space and  $dV$ be the ordinary volume measure on $\mathbb{C}^n$. We denote by $d\mu$ the normalized Gaussian measure on $\mathbb{C}^n$ given by
$$d\mu(z):=\frac{1}{\pi^n}e^{-|z|^2}dV(z).$$
If $z=(z_1,\cdots,z_n)$ and $w=(w_1,\cdots,w_n)$ are points in $\mathbb{C}^n,$ we write
$$z\cdot\overline{w}=\sum_{i=1}^nz_i\overline{w_i},\quad |z|^2=z\cdot\overline{z}.$$

The Fock space $F^2$ is the space of all Gaussian square integrable entire function on $\mathbb{C}^n,$ i.e.,
$$F^2:=L^{2}(\mathbb{C}^n,d\mu)\cap H(\mathbb{C}^n).$$
Here, and elsewhere, $H(\cdot)$ denotes the class of functions holomorphic over the domain specified in the parentheses.
Clearly, $F^2$ is a closed subspace of $L^{2}(\mathbb{C}^n,d\mu)$. So, $F^2$ is a Hilbert space. There is an orthogonal projection $P$ from $L^{2}(\mathbb{C}^n,d\mu)$ onto $F^{2}$, which is given by
$$Pf(z)=\langle f(w), K(w,z)\rangle=\int_{\mathbb{C}^n} f(w) K(z,w)d\mu(w),$$
where $$K_z(w)=K(w,z)=e^{w\cdot\overline{z}}$$ is the reproducing kernel of Fock space $F^{2}.$

For a complex measurable function $f$ on $\mathbb{C}^n$ satisfying suitable growth condition at $\infty$, the Toeplitz operator $T_f$ with symbol $f$ is given by
$$T_f:=PM_f$$
where $M_f$ is the operator of point multiplication by $f$. In general, the symbol function of the Toeplitz operator is unbounded, which can result in an unbounded operator on $F^2$. For the function $f$, we define  $$\mathcal{B}[f](z)=\langle f(w)k(w,z),k(w,z)\rangle,$$
where $k(w,z)=K(w,z)/\sqrt{K(z,z)}.$  We call $\mathcal{B}[f]$ the Berezin transform of $f.$

We now define our symbol space. Given $c>0$, consider the space $\mathcal{D}_c$ of complex measure function $u$ on $\mathbb{C}^n$ such that $u(z)e^{-c|z|^2}$ is essentially bounded on $\mathbb{C}^n$. We define
$$\text{Sym}(\mathbb{C}^n):=\bigcap_{c>0} \mathcal{D}_c.$$  For a measurable  $f$ on $\mathbb{C}^n$, we  write $f\in \varepsilon(\mathbb{C}^n)$, if there is some $c>0$ such that $|f(z)|e^{-c|z|}$ is essentially bounded on $\mathbb{C}^n$. It is clear that
$$\varepsilon(\mathbb{C}^n)\subset \text{Sym}(\mathbb{C}^n)\subset \mathcal{D}_c$$
for $c>0.$  Chose $c=\frac{1}{8},$ then
$$\text{Sym}(\mathbb{C}^n)\subset \mathcal{D}_{\frac{1}{8}}\subset L^{2}(\mathbb{C}^n),d\mu).$$
 Then, for $f,g\in \text{Sym}(\mathbb{C}^n)\cap H(\mathbb{C}^n),$ one can see that  $|fg|^2e^{-\frac{1}{2}|z|^2}$ is essentially bounded on $\mathbb{C}^n$. This implies that
 \begin{equation}\label{E14}
 f\overline{g}\in L^{2}(\mathbb{C}^n,d\mu).
 \end{equation}
See \cite{Bauer1}, for more detail about the symbol space $\text{Sym}(\mathbb{C}^n)$.

In \cite{Brown}, Brown and Halmos considered the commuting Toepitz operators on the Hardy space over the unit disc, as well as characterized all triples of Toeplitz operators $(T_f,T_g,T_h)$ such that $T_fT_g=T_h$. They obtain that the product of two Toeplitz operators is zero if and only if one of them is zero. These theorems are commonly referred to as the Brown-Halmos theorems. Extending these results to the Hilbert spaces of holomorphic functions on more general domains has been one of the central themes of research in the theory of Toeplitz operators in the last few decades.

In the recent paper, Le and Tikaradze \cite{Le} provided a more complete answer to the possible $h$ so that $T_fT_g=T_h$ on the Bergman space over the unit ball under the assumption that $f,g$ are bounded  pluriharmonic and $h$ is sufficiently smooth and bounded. They showed that $T_fT_g=T_h$ if and only if $\overline{f}$ or $g$ is holomorphic. For other studies, refer to \cite{Ahern,Aleman,Axler,Choe,Choe1,Choe2,Choe3,Ding,Guo,Le,Zheng}.

On the other hand, there has not been much progress in the Brown-Halmos theorems on the Fock space. Bauer and Le \cite{Bauer1} studied Algebraic properties and the finite rank problem for Toeplitz operators on the Fock space.
 In \cite{Bauer2}, Bauer and Lee considered the commuting Toeplitz operators with  radial symbols.
 Bauer et al. \cite{Bauer} studied the commuting Toeplitz operators with $\varepsilon(\mathbb{C})-$harmonic symbols.

To date, the Brown-Halmos theorems  on the Fock space remains an open problem.
The motivation of this paper is to provide a generalization of the Brown-Halmos theorems on the Fock space.
In what follows we use the conventional multi-index notation. Thus for an $n-$tuple $i=(i_1,...,i_n)$ of non-negative integers, we write
$$|i|=\sum_{k=1}^n i_k,\quad i!=\prod_{k=1}^n i_k!,\quad z^i=z_1^{i_1}\cdots {z_n}^{i_n} ,\quad\partial^i=\partial_{1}^{i_1}\cdots \partial_{n}^{i_n} $$
where $\partial_j$ denote  partial differentiation with respect to the  $j-$th component.
 We put $$f^*(z)=\overline{f(\overline{z})}$$ for $f\in H(\mathbb{C}^n).$
First, we consider the range of the Berezin transform. Precisely, the first main result is stated as follows.
\begin{theorem}\label{TT1}
Given $f_l,g_l\in H(\mathbb{C}^n)\cap \mathrm{Sym}(\mathbb{C}^n)$ for $l=1,\cdots,N$. Suppose that  $$\mathcal{B}[u]=\sum_{l=1}^Nf_{l}\overline{g_l}\quad\text{with}\quad u\in L^{2}(\mathbb{C}^n,d\mu).$$ Then
$$u(\zeta)=\sum_{l=1}^N g_l^*(\overline{\zeta}-\partial_\zeta)f_l(\zeta).$$
\end{theorem}


As an immediate corollary, we have the following direct generalization of the Brown-Halmos theorems on the Fock space.
\begin{theorem}\label{TT2}
Suppose that $f,g,u,v\in \text{Sym}(\mathbb{C}^n)\cap H(\mathbb{C}^n)$. Write
$$\varphi=f+\overline{g}\quad\text{and}\quad\psi=u+\overline{v}.$$
(a) If $T_{\varphi}T_{\psi}=T_h$ on $F^2$ for some $h\in L^{2}(\mathbb{C}^n,d\mu),$ then
\begin{equation}\label{E11}
h=u\overline{g}+fu+\overline{g}\overline{v}+v^*({\overline{\cdot}}-\partial_\cdot) f(\cdot).
\end{equation}
(b) If $T_{\varphi}T_{\psi}=0$ on $F^2,$  then
\begin{equation*}
u\overline{g}+fu+\overline{g}\overline{v}+v^*({\overline{\cdot}}-\partial_\cdot) f(\cdot)=0.
\end{equation*}
\end{theorem}
We note that we are free to choose $u$ and $g$ that satisfy the  Equation  (\ref{E11}).
From the statement (a) in Theorem \ref{TT2}, there are pluriharmonic function $\varphi$ and $\psi$  such that $$T_{\varphi}T_{\psi}=T_h,\quad h\neq\varphi\psi,$$ see Proposition \ref{l3}. Moreover, there are  non-pluriharmonic functions $\varphi$ and $\psi$ so that $T_{\varphi}T_{\psi}$ is a Toeplitz operator, see Proposition \ref{p1}.

The zero product for  general symbols is a long standing open problem in the area of Toeplitz operators, which has researcher's attempts even for the unit disc. The same problem on the Hardy/Bergman space was considered in \cite{Aleman,Brown,Choe1,Le}.
 In addition, Bauer et al.  \cite{Bauer1} raised the following question in 2015, see \cite[Theorem 4.2, p3047]{Bauer}.
\begin{ques}\label{Q2}
Let $f,g,u,v\in \varepsilon(\mathbb{C})\cap H(\mathbb{C})$. Write
$$\varphi=f+\overline{g}\quad\text{and}\quad\psi=u+\overline{v}.$$ Assume $ T_\varphi T_\psi=0$ on the Fock space. Is it true that $f = 0$ or $g = 0$?
\end{ques}

%
According to the statement (b) in Theorem \ref{TT2}, we obtain the following corollary, which in particular settles the zero product problem for Toeplitz operators with pluriharmonic symbols. The corollary says that the assertion in Question \ref{Q2} is true.  Our result here recover \cite[Theorem 4.2]{Bauer}.
\begin{cor}\label{c3}
Let $f,g,u,v\in \mathrm{Sym}(\mathbb{C}^n)\cap H(\mathbb{C}^n)$. Write
$$\varphi=f+\overline{g}\quad\text{and}\quad\psi=u+\overline{v}.$$ If $T_\varphi T_\psi=0$, then
$\varphi=0$ or $\psi=0$.
\end{cor}

Specifically, we can derive the following corollary, which is of independent interest, is the characterization of the product of the Toeplitz operators of the form
$$\sum_{l=1}^NT_{f_l} T_{\overline{g_l}}$$ where $f_l$ and $g_l$ are functions in $\mathrm{Sym}(\mathbb{C}^n)\cap H(\mathbb{C}^n)$.
\begin{cor}\label{c2}
Suppose that $f_l$ and $g_l$ are functions in $\mathrm{Sym}(\mathbb{C}^n)\cap H(\mathbb{C}^n)$ for $l=1,\cdots,N$. Then
$$\sum_{l=1}^NT_{f_l} T_{\overline{g_l}}=\sum_{l=1}^N T_{g_l^*({\overline{\cdot}}-\partial_\cdot) f_l(\cdot)}.$$
\end{cor}

For $g\in \mathrm{Sym}(\mathbb{C}^n)\cap H(\mathbb{C}^n)$,  it follows from Corollary \ref{c2} and Proposition \ref{p2} that
\begin{equation*}
T_{e^{z\cdot \overline{\eta}}}T_{\overline{g(z)}}=T_{\overline{g(z-\eta)}e^{z\cdot \overline{\eta}}}.
\end{equation*}
Assume $e^{z\cdot \overline{\eta}}\overline{g(z)}$ is the fixed point of the Berezin transform on the Fock space,
it follows from the  injectivity of the Berezin transform  that
$$g(z)=g(z-\eta).$$
This shows that $g$ is a  periodic function with period $\eta.$
Using the above fact, we can immediately obtain the results of the semi-commuting Toeplitz operators with symbols of linear exponential functions. Our result here in the single variable setting reduces to \cite[Section 4-5]{MA}.


Given $T_\varphi$ and $T_\psi$ with $\varphi,\psi\in \text{Sym}(\mathbb{C}^n)$, we denote by
$$[T_{\varphi},T_\psi]:=T_\varphi T_\psi-T_\psi T_\varphi,$$
the commutator of $T_\varphi$ and $T_\psi$. An other direct consequence of Theorem \ref{TT1} is a  improvement of the aforementioned Bauer et al.'s theorem  about commuting Toeplitz operators with harmonic symbols(the several-variable case is open, see \cite{Bauer}). Our result here recover [4, Theorem 4.15].
\begin{theorem}\label{TT3}
Given $f,g,u,v\in \text{Sym}(\mathbb{C}^n)\cap H(\mathbb{C}^n).$ then the following statements are equivalent:\\
 (a) $[T_{f+\overline{g}},T_{u+\overline{v}}]=0$.\\
 (b)  $\mathcal{B}[u\overline{g}-f\overline{v}]=u\overline{g}-f\overline{v}$.\\
 (c)  the following equation is fulfilled:
\begin{equation}\label{E16}
u\overline{g}-f\overline{v}=g^*({\overline{\cdot}}-\partial_\cdot) u(\cdot)-v^*({\overline{\cdot}}-\partial_\cdot) f(\cdot).
\end{equation}
\end{theorem}
Define
$$I_1=(1,\cdots,1)\quad I_{n-1}^*=(1,\ldots,1,0)\quad\text{and}\quad I_n=I_1-I_{n-1}^*=(0,\cdots,0,1)\in \mathbb{C}^n.$$
From Theorem \ref{TT3}, we obtain the following corollary,  which states that the conclusion in \cite[Lemma 3.4]{Bauer1}  is incorrect. It is likely that the authors neglected to restrict $n$ to 1, indicating a clerical error. We would like to alert the reader that the existence of
a holomorphic polynomial $p$ such that $$\mathcal{B}[pK_{-2\pi \mathrm{i}I_n}\overline{K_{I_n}}]=pK_{-2\pi \mathrm{i}I_n}\overline{K_{I_n}}$$
and $p$ is non-constant is a high dimensional phenomenon. So, this clerical error does not impact the primary outcome in \cite{Bauer}.
\begin{cor}\label{c4}
For $n\geq2$. Let $p$ be a non-constant holomorphic polynomial such that $$p(z)=p(z\cdot I_{n-1}^*).$$ Suppose that $f(z)=K_{-2\pi \mathrm{i}I_n}(z)$ and $g(z)=K_{I_n}(z)$. Then $T_{pf}T_{\overline{g}}=T_{pf\overline{g}}$.
\end{cor}

In the next section, we examine the range of the Berezin transform. We obtain that $u$ doesn't need to be a pluriharmonic function if $B[u]=\sum_{l=1}^Nf\overline{g}$ under our  hypothesis, see Proposition \ref{p2}. In section \ref{s3}, we study the Brown-Halmos theorems. We obtain explicit characterizations, see Theorem \ref{t1}. As an application for our results, we complete characterize the zero product of the Toeplitz operators, see Theorem \ref{t2}. Moreover, we provide some examples to explain the  difference between the Brown-Halmos theorems on the Bergman and Fock space, see Proposition \ref{l3} and \ref{p1}.
 These characterizations show that
there are extra cased for the Fock space, which have no analogue on the Bergman space and Hardy space. Finally,
we collect and discuss some problems in section \ref{s4}.

\section{The range of Berezin transform}\label{s2}
The problem to be studied in this section is to consider the range of the Berezin transform on the Fock space.
Given $f_l,g_l\in H(\mathbb{C}^n)\cap \text{Sym}(\mathbb{C}^n)$ for $l=1,\cdots,N$. Consider the equation
\begin{equation}\label{E2}
\mathcal{B}[u]=\sum_{l=1}^Nf_{l}\overline{g_l}\quad\text{with}\quad u\in L^{2}(\mathbb{C}^n,d\mu).
\end{equation}
We define $$\mathcal{F}(ue^{-|\cdot|^2})(z)=\frac{1}{\pi^n}\int_{\mathbb{C}^n}
u(\zeta)e^{-|\zeta|^2}e^{\mathrm{i}\Re(z\cdot\overline{\zeta})}dV(\zeta),$$
where $\mathcal{F}$ stands for the Fourier transform and $\Re(z\cdot\overline{\zeta})$ is the real part of $z\cdot\overline{\zeta}$, see \cite{Bauer} for more information. Recall that $$f^*(z)=\overline{f(\overline{z})}$$ for $f\in H(\mathbb{C}^n).$ The solution $u$ was obtained in \cite{Bauer}. For the reader's
convenience, we include the formula and its proof here.
\begin{lemma}\label{l1}
Let $$u(\zeta)=e^{|\zeta|^2}\mathcal{F}^{-1}\bigg[e^{-\frac{|z|^2}{4}}
Q(\frac{\mathrm{i}z}{2},\frac{\mathrm{i}\overline{z}}{2})\bigg](\zeta),\quad Q(z,w)=\sum_{l=1}^N f_{l}(z)g_l^*(w).$$
Then $u$ above is the unique solution of (\ref{E2}).
\end{lemma}
\begin{proof}
Since $\mathcal{B}[u]=\sum_{l=1}^Nf_{l}\overline{g_l}$, then
$$\sum_{l=1}^Nf_{l}\overline{g_l}=\frac{e^{-|z|^2}}{\pi^n}\int_{\mathbb{C}^n}u(\zeta)\exp\{z\cdot\overline{\zeta}+\overline{z}\cdot\zeta-|\zeta|^2\}dV(\zeta).$$
Complexifying both sides, we have
\begin{align*}
Q(z,w)&=\sum_{l=1}^Nf_l(z)g_l^*(w)\\
&=\frac{e^{-z\cdot w}}{\pi^n}\int_{\mathbb{C}^n}u(\zeta)\exp\{z\cdot\overline{\zeta}+w\cdot\zeta-|\zeta|^2\}dV(\zeta)
\end{align*}
for $z,w\in \mathbb{C}^n$. This equivalent to
\begin{align*}
e^{-\frac{|z|^2}{4}}Q(\frac{\mathrm{i}z}{2},\frac{\mathrm{i}\overline{z}}{2})&
=\frac{1}{\pi^n}\int_{\mathbb{C}^n}u(\zeta)e^{-|\zeta|^2}e^{\mathrm{i}\Re(z\cdot\overline{\zeta})}dV(\zeta)\\
&:=\mathcal{F}(ue^{-|\cdot|^2})(z),
\end{align*}
where $\mathcal{F}$ is the Fourier transform. Note that the function $z\rightarrow Q(\frac{\mathrm{i}z}{2},\frac{\mathrm{i}\overline{z}}{2})$ belongs to the symbol space. Applying the inverse Fourier transform $\mathcal{F}^{-1}$  on $L^{2}(\mathbb{C}^n,dV)$, we thus obtain
$$u(\zeta)=e^{|\zeta|^2}\mathcal{F}^{-1}\bigg[e^{-\frac{|z|^2}{4}}Q(\frac{\mathrm{i}z}{2},\frac{\mathrm{i}\overline{z}}{2})\bigg](\zeta)$$ for
$\zeta\in \mathbb{C}^n$. This implies that $u$ above is the unique solution of (\ref{E2}).
\end{proof}

In the rest of this paper, we use the following notations  $$\mathrm{C}_{i}^{l}=\prod_{k=1}^n\mathrm{C}_{i_k}^{l_k}\quad, \quad \mathrm{C}_{i_k}^{l_k}=\frac{i_k!}{l_k!(i_k-l_k)!},\quad |i-l|=\sum_{k=1}^n (i_k-l_k)$$
for any multi-index $i$ and $l$. We say that $i\geq l$ if $i_k\geq l_k$ for each $k$. It is easy to see that
 \begin{align*}
(\overline{z}+\partial_z)^if(z)&=\prod_{k=1}^n (\overline{z_k}+\partial_{z_k})^{i_k}f(z_1,\cdots,z_n)\nonumber\\
&=\prod_{k=2}^n (\overline{z_k}+\partial_{z_k})^{i_k}\bigg(\sum_{l_1=0}^{i_1}\mathrm{C}_{i_1}^{l_1} \overline{z_1}^{l_1}\partial_{1}^{i_1-l_1} f(z) \bigg)\nonumber\\
&=\prod_{k=1}^n  \bigg(\sum_{l_k=0}^{i_k}\mathrm{C}_{i_k}^{l_k} \overline{z_k}^{l_k}\partial_{k}^{i_k-l_k} f(z) \bigg).
\end{align*}
So we obtain that
\begin{equation}\label{E8}
(\overline{z}+\partial_z)^if(z):=\sum_{l\leq i} \mathrm{C}_{i}^{l}\overline{ z}^l \partial_{z}^{i-l}f(z).
\end{equation}
In fact, the unique solution of (\ref{E2}) above can be made more explicit via the next theorem.
\begin{theorem}\label{l2}
If $u$ above is the solution of (\ref{E2}), then
$$u(\zeta)=\sum_{l=1}^Ng_l^*(\overline{\zeta}-\partial_\zeta)f_l(\zeta)\in L^2(\mathbb{C}^n,d\mu).$$
\end{theorem}
\begin{proof}
By Lemma \ref{l1}, $$u(\zeta)=e^{|\zeta|^2}\mathcal{F}^{-1}\bigg[e^{-\frac{|z|^2}{4}}Q(\frac{\mathrm{i}z}{2},\frac{\mathrm{i}\overline{z}}{2})\bigg](\zeta)$$
for $\zeta\in \mathbb{C}^n$. Let $g_l(z)=\sum_{i} a_{il} z^i.$ By a change of variables and standard calculus computations,
\begin{align*}
e^{-|\zeta|^2}u(\zeta)
&=\sum_{l=1}^N\frac{1}{\pi^n}\int_{\mathbb{C}^n}  e^{-|z|^2}f_l(z)\overline{g_l(-z)}e^{-\mathrm{i}\Re(-2\mathrm{i}z\cdot\overline{\zeta})}dV(z)\\
&=\sum_{l=1}^N\frac{1}{\pi^n}\int_{\mathbb{C}^n}  f_l(z)\overline{g_l(-z)} e^{\overline{z}\cdot\zeta-z\cdot\overline{\zeta}-|z|^2}dV(z)\\
&=\sum_{l=1}^N\sum_i \overline{a_{il}}\frac{1}{\pi^n}\int_{\mathbb{C}^n} f_l(z)  \overline{(-z)^i} e^{\overline{z}\cdot\zeta-z\cdot\overline{\zeta}-|z|^2}dV(z).
\end{align*}
It follows that
\begin{align*}
e^{-|\zeta|^2}u(\zeta)&=\sum_{l=1}^N\sum_i \overline{a_{il}}\frac{1}{\pi^n}\int_{\mathbb{C}^n} (-1)^{|i|}f_l(z)  \partial_\zeta^i(e^{\overline{z}\cdot\zeta-z\cdot\overline{\zeta}-|z|^2})dV(z)\\
&=\sum_{l=1}^N\sum_i \overline{a_{il}}(-1)^{|i|}\partial_\zeta^i\bigg[\frac{1}{\pi^n}\int_{\mathbb{C}^n} f_l(z)e^{\overline{z}\cdot\zeta-z\cdot\overline{\zeta}-|z|^2}dV(z)\bigg]\\
&=\sum_i \overline{a_{il}} (-1)^{|i|}\partial^i_\zeta[f_l(\zeta)e^{-|\zeta|^2}].
\end{align*}
It is easy to check that
$$\sum_i \overline{a_{il}} (-1)^{|i|}\partial^i_\zeta[f_l(\zeta)e^{-|\zeta|^2}]=\sum_i \overline{a_{il}} e^{-|\zeta|^2}(-1)^{|i|} (\partial_\zeta-\overline{\zeta})^if_l(\zeta).$$
So we can see that
$$u(\zeta)=\sum_{l=0}^Ng_l^*(\overline{\zeta}-\partial_\zeta)f_l(\zeta).$$

Now, we will prove that $\mathcal{B}[u]=\sum_{i=0}^N f_l\overline{g_l}$. Since the Berezin transform is linear, we only need to consider the simplest case:
$$\mathcal{B}[g^*(\overline{z}-\partial_z)f(z)](\zeta)=f(\zeta)\overline{g(\zeta)}.$$   It follows from (\ref{E8}) that
\begin{align*}
&\quad e^{|\zeta|^2}\mathcal{B}[g^*(\overline{z}-\partial_z)f(z)](\zeta)\\&=\frac{1}{\pi^n} \int_{\mathbb{C}^n}[g^*(\overline{z}-\partial_z)f(z)]e^{\overline{z}\cdot\zeta+z\cdot\overline{\zeta}-|z|^2}dV(z)\\
&=\sum_i \overline{a_i}\frac{1}{\pi^n} \int_{\mathbb{C}^n} [(\overline{z}-\partial_z)^if(z)]e^{\overline{z}\cdot\zeta+z\cdot\overline{\zeta}-|z|^2}dV(z)\\
&=\sum_i \overline{a_i}\bigg(\sum_{j\leq i}\mathrm{C}_{i}^j\frac{1}{\pi^n}\int_{\mathbb{C}^n}(-1)^{|i-j|}
\overline{z}^j[\partial_z^{i-j}f(z)]e^{\overline{z}\cdot\zeta+z\cdot\overline{\zeta}-|z|^2}dV(z)\bigg).
\end{align*}
By a simple computation,
\begin{align*}
&\quad\sum_{j\leq i}\mathrm{C}_{i}^j(-1)^{|i-j|}\frac{1}{\pi^n}\int_{\mathbb{C}^n}
\overline{z}^j[\partial_z^{i-j}f(z)]e^{\overline{z}\cdot\zeta+z\cdot\overline{\zeta}-|z|^2}dV(z)\\
&=\sum_{j\leq i}\mathrm{C}_{i}^j\frac{1}{\pi^n}\int_{\mathbb{C}^n}(-1)^{|i-j|}
[\partial_z^{i-j}f(z)]\partial_\zeta^j [e^{\overline{z}\zeta+z\overline{\zeta}-|z|^2}]dV(z)\\
&=\sum_{j\leq i}\mathrm{C}_{i}^j(-1)^{|i-j|}\partial_\zeta^j\bigg(\frac{1}{\pi^n}\int_{\mathbb{C}}
[\partial_z^{i-j}f(z)]e^{\overline{z}\cdot\zeta+z\cdot\overline{\zeta}-|z|^2}dV(z)\bigg)\\
&=\sum_{j\leq i}\mathrm{C}_{i}^j(-1)^{|i-j|}\partial_\zeta^j \bigg[e^{|\zeta|^2}\partial_\zeta^{i-j}f(\zeta)\bigg]\\
&=\sum_{j\leq i}\mathrm{C}_{i}^j(-1)^{|i-j|} e^{|\zeta|^2}(\overline{\zeta}+\partial_\zeta)^j[\partial_\zeta^{i-j}f(\zeta)].
\end{align*}
So we have
\begin{align}\label{E7}
\mathcal{B}[g^*(\overline{z}-\partial_z)f(z)](\zeta)&=\sum_i \overline{a_i} \sum_{j\leq i}\mathrm{C}_{i}^j(-1)^{|i-j|}(\overline{\zeta}+\partial_\zeta)^j[\partial_\zeta^{i-j}f(\zeta)]\nonumber\\
&=\sum_i \overline{a_i}  \sum_{j\leq i}\mathrm{C}_{i}^j(\overline{\zeta}+\partial_\zeta)^j[(-\partial_\zeta)^{i-j}f(\zeta)].
\end{align}
It is clear that
\begin{align*}
(\overline{\zeta}+\partial_\zeta)^j[(\partial_\zeta)^{i-j}f(\zeta)]&=\sum_{k\leq j} \mathrm{C}_j^k\overline{\zeta}^k \bigg\{\partial_{\zeta}^{j-k}[\partial_\zeta^{i-j}f(\zeta)]\bigg\}
\\&=\sum_{k\leq j} \mathrm{C}_j^k\overline{\zeta}^k \bigg\{\partial_{\zeta}^{i-j} [\partial_\zeta^{j-k}f(\zeta)]\bigg\}\\
&=\partial_\zeta^{i-j}[(\overline{\zeta}+\partial_\zeta)^jf(\zeta)].
\end{align*}
This fact with (\ref{E7}) gives
\begin{align*}
&\quad\mathcal{B}[g^*(\overline{z}-\partial_z)f(z)](\zeta)\\
&=\sum_i \overline{a_i}   \sum_{j\leq i} \mathrm{C}_{i}^j (-1)^{|i-j|}\partial_\zeta^{i-j}[(\overline{\zeta}+\partial_\zeta)^jf(\zeta)] \\ 
&=\sum_i \overline{a_i}   \sum_{j\leq i} \mathrm{C}_{i}^j(-\partial_\zeta)^{i-j}[(\overline{\zeta}+\partial_\zeta)^jf(\zeta)] \\
&=\sum_i \overline{a_i} \overline{\zeta}^i f(\zeta)\\
&=f(\zeta) \overline{g(\zeta)}.
\end{align*}
This shows that $B[u]=\sum_{l=1}^N f_l\overline{g_l}$,
which completes the proof.
\end{proof}

Recall that $$I_1=(1,1\cdots,1)\in \mathbb{C}^n.$$ It is clear that
$$|I_1|^2=n,\quad I_1^i=1^{|i|}=1$$  for every multi-index $i.$
There is a function $f(z)=e^{z\cdot I_1}$ so that
\begin{equation}\label{E9}
\partial^i_z f=f
\end{equation} for every multi-index $i.$ It is easy to see that
$$|z\cdot I_1|\leq n|z|.$$ So, $f\in \mathrm{Sym}(\mathbb{C}^n)\cap H(\mathbb{C}^n).$
From Theorem \ref{l2}, we obtain the following interesting result, which shows that there is a non-trivial solution $u$. This means that
both $f$ and $g$ are non-constants functions in (\ref{E2}), or equivalently, $u$ is a non-pluriharmonic function.
\begin{prop}\label{p2}
Suppose that $u\in L^{2}(\mathbb{C}^n,d\mu)$ and $u$ is the solution  of (\ref{E2}). Let $f(z)=e^{z\cdot I_1}$, then
$$u(z)=e^{z\cdot I_1}\overline{g(z-I_1)}.$$
\end{prop}
\begin{proof}
By Theorem \ref{l2}, $u(z)=g^*(\overline{z}-\partial_z)f(z).$ Let $g(z)=\sum_i a_i z^i.$ Since $f(z)=e^{z\cdot I_1},$ then, by (\ref{E9}), we obtain
\begin{align*}
u(z)&=\sum_i \overline{a_i}(\overline{z}-\partial_z) e^{z\cdot I_1}\\
&=\sum_i \overline{a_i} \sum_{j\leq i} \mathrm{C}_{i}^j \overline{z}^j [(-\partial_z)^{i-j}e^{z\cdot I_1}]\\
&=\sum_i \overline{a_i} \sum_{j\leq i} \mathrm{C}_{i}^j \overline{z}^j {(-1)}^{|i-j|}e^{z\cdot I_1}.
\end{align*}
So we have
\begin{align*}
u(z)
&=\sum_i \overline{a_i} e^{z\cdot I_1}(\overline{z}-I_1)^i\\
&=\overline{g(z-I_1)}e^{z\cdot I_1}.
\end{align*}

Let's compute the Berezin transform of $\overline{g(z-I_1)}e^{z\cdot I_1}$. By the definition of the Berezin transfom,
\begin{align*}
e^{|\zeta|^2}\mathcal{B}[u](\zeta)&=\sum_i \overline{a_i}\int_{\mathbb{C}^n} (\overline{z}-I_1)^i f(z) e^{\overline{z}\cdot \zeta+z\cdot \overline{\zeta}}d\mu(z)\\
&=\sum_i \overline{a_i} \sum_{j\leq i}\mathrm{C}_{i}^j(-I_1)^{i-j}
\int_{\mathbb{C}^n} \overline{z}^j f(z) e^{\overline{z}\cdot \zeta+z\cdot \overline{\zeta}}d\mu(z).
\end{align*}
It is easy to see that
\begin{align*}
\int_{\mathbb{C}^n} \overline{z}^j f(z) e^{\overline{z}\cdot \zeta+z\cdot \overline{\zeta}}d\mu(z)&=\partial_\zeta^j[f(\zeta)e^{|\zeta|^2}].
\end{align*}
Since $f(z)=e^{z\cdot I_1}$, then we have
\begin{align*}
\int_{\mathbb{C}^n} \overline{z}^j f(z) e^{\overline{z}\cdot \zeta+z\cdot \overline{\zeta}}d\mu(z)&=(\overline{\zeta}+I_1)^je^{\zeta\cdot(\overline{\zeta}+I_1)}.
\end{align*}
It follows that
\begin{align*}
\mathcal{B}[u](\zeta) &=\sum_i \overline{a_i} \sum_{j\leq i}\mathrm{C}_{i}^j(-I_1)^{i-j}(\overline{\zeta}+I_1)^j e^{\zeta\cdot I_1}\\
&=\sum_i \overline{a_i} \overline{\zeta}^i e^{\zeta\cdot I_1} \\
&=\overline{g(\zeta)}e^{\zeta\cdot I_1}.
\end{align*}
This implies that $\mathcal{B}[u]=f\overline{g}$, which completes the proof.
\end{proof}


We close this section with a remark on the range of the Berezin transform.
\begin{remark}\label{r1}
In Proposition \ref{p2}, we consider the case of $f(z)=e^{z\cdot I_1}$. In fact, if we set $f(z)=e^{z\cdot \overline{\eta}}$, then we obtain
$$u(z)=\overline{g(z-\eta)}e^{z\cdot \overline{\eta}},$$
where $u$ is the solution  of (\ref{E2}). So, many non-pluriharmonic functions $u$ with $\mathcal{B}[u]=f\overline{g}$ can be constructed in $L^2(\mathbb{C}^n,d\mu).$
\end{remark}
\section{Brown-Halmos type results}\label{s3}
In this section, we aim to characterize the Brown-Halmos theorems on the Fock space. We recall that
$T_fT_{\overline{g}}=T_{f\overline{g}}$ if and only if $$\mathcal{B}[T_fT_{\overline{g}}]=\mathcal{B}[T_{f\overline{g}}]=f\overline{g}$$
for $f,g\in\text{Sym}(\mathbb{C}^n)\cap H(\mathbb{C}^n).$ We now proceed to prove the statement (a) in Theorem \ref{TT2}.
\begin{theorem}\label{t1}
  Write
$$\varphi=f+\overline{g}\quad\text{and}\quad\psi=u+\overline{v}$$
where $f,g,u,v\in \text{Sym}(\mathbb{C}^n)\cap H(\mathbb{C}^n)$.
Suppose $h\in L^2(\mathbb{C}^n,d\mu),$ then $T_\varphi T_\psi=T_h$ if and only if
$$h=u\overline{g}+fu+\overline{g}\overline{v}+v^*({\overline{\cdot}}-\partial_\cdot) f(\cdot).$$
\end{theorem}
\begin{proof}
Trivially, $\mathcal{B}[T_\varphi T_\psi]=\mathcal{B}[h]$ if $T_\varphi T_\psi=T_h$. Then
$$\mathcal{B}[T_\varphi T_\psi]=fu+f\overline{v}+\mathcal{B}[u\overline{g}]+\overline{g}\overline{v}=\mathcal{B}[h].$$
It follows that
$$\mathcal{B}[h-u\overline{g}-fu-\overline{g}\overline{v}]=f\overline{v}.$$
From (\ref{E14}), one can see that $$u\overline{g}-fu-\overline{g}\overline{v}\in  L^2(\mathbb{C}^n,d\mu).$$
By Theorem \ref{l2}, we obtain that
$$h(z)-[u\overline{g}](z)-[fu](z)-[\overline{g}\overline{v}](z)=v^*(\overline{z}-\partial_z)f(z)\in L^2(\mathbb{C}^n,d\mu).$$
So we have $$h=u\overline{g}+fu+\overline{g}\overline{v}+v^*({\overline{\cdot}}-\partial_\cdot) f(\cdot)\in L^2(\mathbb{C}^n,d\mu).$$

On the other hand, we have
\begin{equation}\label{E3}
T_h=T_{u\overline{g}+fu+\overline{g}\overline{v}}+T_{v^*({\overline{\cdot}}-\partial_\cdot) f(\cdot)}.
\end{equation}
By the definition of the Toeplitz operator,
\begin{equation}\label{E4}
T_{\varphi}T_{\psi}=T_{u\overline{g}+fu+\overline{g}\overline{v}}+T_{f}T_{\overline{v}}.
\end{equation}
It is easy to see that $B[T_fT_{\overline{v}}]=f\overline{v}.$  Then, by Lemma \ref{l1}  and Theorem \ref{l2}, we conclude that
$$\mathcal{B}[v^*(\overline{\cdot}-\partial_\cdot)f(\cdot)]=f\overline{v}.$$
It follows that $$B[T_fT_{\overline{v}}]=\mathcal{B}[v^*(\overline{\cdot}-\partial_\cdot)f(\cdot)]=f\overline{v},$$
which implies that $$T_fT_{\overline{v}}=T_{v^*({\overline{\cdot}}-\partial_\cdot) f(\cdot)}.$$
This fact with (\ref{E3}-\ref{E4}) shows that
$$T_{\varphi}T_{\psi}=T_h.$$ This completes the proof.
\end{proof}
As an immediate corollary, we now obtain (b) in Theorem \ref{TT2}.
\begin{cor}\label{c1}
Write
$$\varphi=f+\overline{g}\quad\text{and}\quad\psi=u+\overline{v}$$
where $f,g,u,v\in \text{Sym}(\mathbb{C}^n)\cap H(\mathbb{C}^n)$.
Then $T_\varphi T_\psi=0$ if and only if
\begin{equation*}
u\overline{g}+fu+\overline{g}\overline{v}+v^*({\overline{\cdot}}-\partial_\cdot) f(\cdot)=0.
\end{equation*}
\end{cor}
\begin{proof}
We put $h=0$ in Theorem \ref{t1}, then we get the result.
\end{proof}

Setting $g=0$ and $u=0$ in Theorem \ref{t1},   we obtain  Corollary \ref{c2}.
So we omit the proof of Corollary \ref{c2} here.
We can now prove the Corollary \ref{c3}, which demonstrates that the assertion regarding the zero product of Toeplitz operators is  true.
\begin{theorem}\label{t2}
Given $f,g,u,v\in \mathrm{Sym}(\mathbb{C}^n)\cap H(\mathbb{C}^n).$ Write $$\varphi=f+\overline{g}\quad\text{and}\quad \psi=u+\overline{v}.$$ Then $T_\varphi T_\psi=0$ if and only if
 $\varphi=0$ or $\psi=0$.
\end{theorem}
\begin{proof}
The sufficiency is clear, we now prove the necessity.
Since $T_\varphi T_\psi=0$, then
$$u\overline{g}+fu+\overline{g}\overline{v}+v^*({\overline{\cdot}}-\partial_\cdot) f(\cdot)=0$$
by Corollary \ref{c1}. We will complexify the above equation to get
\begin{align}\label{EEE1}
u(z)g^*(w)+f(z)u(z)+g^*(w)v^*(w)+v^*(w-\partial_z) f(z)=0.
\end{align}
Setting $w=0$ in (\ref{EEE1}), we have
\begin{align}\label{EEE2}
u(z)g^*(0)+f(z)u(z)+g^*(0)v^*(0)+v^*(-\partial_z) f(z)=0.
\end{align}
Let $v(z)=\sum_\alpha b_\alpha z^\alpha.$ It is easy to check that
\begin{align*}
v^*(-\partial_z)f(z)&=\sum_{\alpha}\overline{b_\alpha}(-1)^{|\alpha|}\partial_z^\alpha\langle f, K_z(w)\rangle\\
&=\sum_{\alpha}\overline{b_\alpha}(-1)^{|\alpha|}\langle \overline{w}^\alpha f(w),K_z(w)\rangle\\
&=\langle \overline{v(-w)}f(w), K_z(w)\rangle\\
&=T_{\overline{v(-z)}}f(z).
\end{align*}
This together with (\ref{EEE2}) shows that
\begin{align*}
u(z)g^*(0)+f(z)u(z)+g^*(0)v^*(0)+T_{\overline{v(-z)}}f(z)=0.
\end{align*}
 That is,
\begin{align}\label{EEE3}
T_{\overline{v(-z)}}f(z)+u(z)\overline{g(0)}+u(z)f(z)+\overline{g(0)v(0)}=0.
\end{align}

For any $\eta\in \mathbb{C}$, we have
\begin{align*}
T_{K_\eta(z)}T_{\overline{v(-z)}}f(z)+K_\eta(z)\bigg(u(z)\overline{g(0)}+u(z)f(z)+\overline{g(0)v(0)}\bigg)=0.
\end{align*}
Using Corollary \ref{c2}, Proposition \ref{p2} and its remark, we obtain
$$T_{K_\eta(z)}T_{\overline{v(-z)}}=T_{\overline{v(\eta-z)}K_\eta(z)}.$$ This clearly implies
\begin{align*}
T_{\overline{v(\eta-z)}f(z)}K_\eta(z) +K_\eta(z)\bigg(u(z)\overline{g(0)}+u(z)f(z)+\overline{g(0)v(0)}\bigg)=0.
\end{align*}
Consequently,
\begin{align*}
0&=\langle T_{\overline{v(\eta-z)}f(z)}K_\eta(z),K_\eta(z)\rangle+
\langle K_\eta(z)\bigg(u(z)\overline{g(0)}+u(z)f(z)+\overline{g(0)v(0)}\bigg),K_\eta(z)\rangle\\
&=\bigg\langle \bigg(\overline{v(\eta-z)}f(z)+u(z)\overline{g(0)}+u(z)f(z)+\overline{g(0)v(0)}\bigg)K_\eta(z)  ,K_\eta(z)\bigg\rangle\\
&=e^{|\eta|^2}\mathcal{B}[\overline{v(\eta-z)}f(z)+u(z)\overline{g(0)}+u(z)f(z)+\overline{g(0)v(0)}](\eta).
\end{align*}
Using the linearity and injectivity of the Berezin transform, we conclude that
$$\overline{v(\eta-z)}f(z)+u(z)\overline{g(0)}+u(z)f(z)+\overline{g(0)v(0)}=0$$
for $\eta\in \mathbb{C}$. The above equation implies that $\overline{v(\eta-z)}$ is a constant or
$f=0.$

Assume $f=0$. Then $\varphi\psi=0$ since $T_\varphi T_\psi=0$ and $\varphi=\overline{g}$.
If $\overline{v(\eta-z)}$ is a constant, then $\overline{v(\eta-z)}=\overline{v(\eta)}.$
Since $\eta$ is arbitrary, so $v$ is a constant. It is easy to see that
$$\mathcal{B}[T_{f+\overline{g}}T_{u+c}]=(f+\overline{g})(u+c)$$
for $\overline{v}=c.$
Thus, $\varphi\psi=0$ since $T_\varphi T_\psi=0$ and $\psi=u+c$. This completes the proof of the
theorem.
\end{proof}

In fact, using Proposition \ref{p2}, it is easy to see that
\begin{equation}\label{E18}
T_{e^{z\cdot I_1}}T_{\overline{v(z)}}=T_{\overline{v(z-I_1)}e^{z\cdot I_1}}.
\end{equation}
 Next, we aim to illustrate the distinction between the Brown-Halmos theorems on the Bergman and Fock space using examples derived from (\ref{E18}).  
\begin{prop}\label{l3}
Suppose that $v(z)=e^{(z+I_1)\cdot{I_1}}$ and $f(z)=e^{z\cdot I_1}$. Then $$T_{f(z)}T_{\overline{v(z)}}=T_{e^{z\cdot I_1+\overline{z}\cdot{I_1}}}\neq T_{f(z)\overline{v(z)}}$$ on the Fock space.
\end{prop}
\begin{proof}
For any multi-index $l$ with $|l|\geq 1,$  we have
\begin{align*}
T_{\overline{v(z)}}z^l&=\langle w^l,e^{(w\cdot I_1+w\cdot \overline{z}+n}\rangle\\
&=\langle w^l,e^{(w\cdot ( I_1+ \overline{z})+n}\rangle\\
&=( I_1+ z)^l e^n.
\end{align*}
Here, we use the fact that $I_1\cdot I_1=n$.
It follows that
\begin{equation}\label{E10}
T_{f(z)}T_{\overline{v(z)}}z^l=( I_1+ z)^l e^{n+z\cdot I_1}=( I_1+ z)^l e^{(z+I_1)\cdot I_1}.
\end{equation}
A similar argument shows that $$T_{e^{z\cdot I_1+\overline{z}\cdot{I_1}}}z^l=( I_1+ z)^l e^{(z+I_1)\cdot I_1}.$$
This together with (\ref{E10}) yields $T_{f(z)}T_{\overline{v(z)}}=T_{e^{z\cdot I_1+\overline{z}\cdot{I_1}}}$. On the other hand, it is easy to check that
$$e^{z\cdot I_1}e^{(\overline{z}+I_1)\cdot{I_1}}=e^{z\cdot I_1+\overline{z}\cdot I_1+n}\neq e^{z\cdot
I_1+\overline{z}\cdot{I_1}},$$ which shows that $h\neq f\overline{v}.$ The proof is complete.
\end{proof}
From Proposition \ref{l3}, we obtain the following result which shows that
there are functions $\varphi,\psi,h$ so that $$T_\varphi T_\psi=T_h\neq T_{\varphi\psi},$$
where $\varphi$ and $\psi$ are non-pluriharmonic functions.
 This characterization suggests that there are extreme cases for Fock space.
\begin{prop}\label{p1}
Let $f(z)=e^{z\cdot I_1}$. Suppose that $h,v,g$ are functions 
in  $H(\mathbb{C}^n)$ and $v\in \mathrm{Sym}(\mathbb{C}^n)$ .
Write  \begin{equation*}
k(z)=\overline{h(z)}\overline{v(z-I_1)}f(z)g(z)
\end{equation*}
 Then $T_{\overline{h}f}T_{\overline{v}g}=T_k.$
\end{prop}
\begin{proof}
By the definition of the Toeplitz operator, 
$$T_{\overline{h}f}T_{\overline{v}g}=T_{\overline{h}}T_{f}T_{\overline{v}}T_{g}.$$
It follows from Proposition \ref{l3} that 
$$T_{f}T_{\overline{v}}=T_{\overline{v(z-I_1)}f(z)},$$
which means that 
$$T_{\overline{h(z)}}T_{f(z)}T_{\overline{v(z)}}T_{g(z)}
=T_{\overline{h(z)}\overline{v(z-I_1)}f(z)g(z)}.$$  
This yields $T_{\overline{h}f}T_{\overline{v}g}=T_k,$ as required.
\end{proof}

Next, we consider a pair of symbol functions $\varphi,\psi\in \mathrm{Sym}(\mathbb{C}^n)$ of the form
$$\varphi=f+\overline{g}\quad\text{and}\quad \psi=u+\overline{v}$$
where $f,g,u,v\in \mathrm{Sym}(\mathbb{C}^n)\cap H(\mathbb{C}^n).$
A direct calculation yields
\begin{equation}\label{E17}
\mathcal{B}\bigg[[T_\varphi,T_\psi]\bigg]=\mathcal{B}[u\overline{g}-f\overline{v}]+
f\overline{v}-u\overline{g}.
\end{equation}
It is clear that $[T_\varphi,T_\psi]=0$ if and only if $\mathcal{B}[u\overline{g}-f\overline{v}]=u\overline{g}-f\overline{v}.$ We now pause to characterize commuting Toeplitz operators with pluriharmonic symbols, which shows Theorem \ref{TT3}.
\begin{theorem}\label{t3}
Given $f,g,u,v\in \text{Sym}(\mathbb{C}^n)\cap H(\mathbb{C}^n).$ Then $[T_{f+\overline{g}},T_{u+\overline{v}}]=0$ if and only if the following equation  is fulfilled:
\begin{equation*}
u\overline{g}-f\overline{v}=g^*({\overline{\cdot}}-\partial_\cdot) u(\cdot)-v^*({\overline{\cdot}}-\partial_\cdot) f(\cdot).
\end{equation*}
\end{theorem}
\begin{proof}
From (\ref{E17}), $[T_{f+\overline{g}},T_{u+\overline{v}}]=0$ implies that $u\overline{g}-f\overline{v}$ is a fixed point of the Berezin transform. This fact with Theorem \ref{TT1} shows that
$$u\overline{g}-f\overline{v}=g^*({\overline{\cdot}}-\partial_\cdot) u(\cdot)-v^*({\overline{\cdot}}-\partial_\cdot) f(\cdot).$$

Conversely, assume the equation (\ref{E16}) holds. By Theorem \ref{TT1} again, on can see that
$$\mathcal{B}[u\overline{g}-f\overline{v}]=u\overline{g}-f\overline{v}.$$ This gives
$[T_{f+\overline{g}},T_{u+\overline{v}}]=0$, which completes the proof.
\end{proof}
Recall that $$I_{n-1}^*=(1,\ldots,1,0)\quad\text{and}\quad I_n=I_1-I_{n-1}^*=(0,\cdots,0,1)\in \mathbb{C}^n.$$ It is easy to see that $I^*_{n-1}\cdot I_n=0.$ We now apply the above theorem to  prove the Corollary \ref{c4}.
\begin{cor}
For $n\geq2$. Let $p$ be a holomorphic polynomial  such that $$p(z)=p(z\cdot I_{n-1}^*).$$ Suppose that $f(z)=e^{2\pi \mathrm{i}z_n}$ and $g(z)=e^{z_n}$. Then $T_{pf}T_{\overline{g}}=T_{pf\overline{g}}$.
\end{cor}
\begin{proof}
We compute the Berezin transform
\begin{align*}
\mathcal{B}[pf\overline{g}](z)&=e^{-|z|^2}\langle p(w)f(w)K_z,g(w)K_z\rangle\\
&=e^{-|z|^2}\langle p(w)e^{2\pi \mathrm{i}w\cdot I_n}K_z,K_{z+I_n}\rangle.
\end{align*}
 It follows that
\begin{align*}
\mathcal{B}[pf\overline{g}](z)&=p(z+I_n)e^{2\pi \mathrm{i}(z+I_n)\cdot I_n}e^{(z+I_n)\cdot \overline{z}-|z|^2}\\
&=p(z\cdot I^*_{n-1})e^{2\pi \mathrm{i}(z_n+1)+\overline{z_n}}\\
&=p(z)f(z)\overline{g(z)}.
\end{align*}
This shows that $T_{pf}T_{\overline{g}}=T_{pf\overline{g}}$ by Theorem \ref{t3}. The proof is complete.
\end{proof}
In Corollary \ref{c4}, one can see that
$$f(z)=e^{z\cdot (2\pi \mathrm{i}I_n)}=K_{\overline{2\pi \mathrm{i}I_n}}(z) \quad\text{and}\quad g(z)=e^{z\cdot I_n}=K_{I_n}(z) .$$ In \cite[Lemma 3.4]{Bauer1}, Bauer et al. claimed that $p$ must be a constant
if Corollary \ref{c4} holds. However, we  have discovered that\cite[Lemma 3.4]{Bauer1} is not valid for $n\geq 2$.
We suspect that the authors overlooked the need to restrict $n$ to 1. So this is a clerical error. \begin{exap}
Let $n=2$. Suppose that
 $$p(z)=P(z\cdot I_1^*)=z_1\quad\text{and}\quad f(z)=e^{z\cdot (2\pi \mathrm{i}I_2)} \quad\text{and}\quad g(z)=e^{z\cdot I_2}.$$
Clearly, $p(z)=z^\beta$ with $\beta=I_2=(0,1).$
For any multi-index $\alpha$,  we have
\begin{align*}
T_{p(z)f(z)\overline{g(z)}}z^\alpha&=\langle p(w)f(w)w^\alpha , g(w)e^{w\cdot \overline{z}}\rangle\\
&=\langle p(w)f(w)w^\alpha , e^{w\cdot (\overline{z}+I_2)}\rangle\\
&=p(z+I_2)f(z+I_2)(z+I_2)^\alpha\\
&=z^{\beta} f(z)(z+I_2)^\alpha.
\end{align*}
So we obtain
\begin{align*}
T_{p(z)f(z)\overline{g(z)}}z^\alpha=p(z)f(z)(z+I_2)^\alpha.
\end{align*}
A similar argument shows that
\begin{align*}
T_{\overline{e^{z\cdot I_2}}}z^\alpha&=\langle w^\alpha, e^{w\cdot (\overline{z}+I_2)} \rangle\\
&=(z+I_2)^\alpha.
\end{align*}
It follows that $$T_{p(z)f(z)}T_{\overline{g(z)}}z^\alpha=T_{p(z)f(z)\overline{g(z)}}z^\alpha.$$
Since the holomorphic polynomials is dense in $F^2$, so we obtain $T_{pf\overline{g}}=T_{pf}T_{\overline{g}}$.
\end{exap}

\section{Open problem}\label{s4}
In this final section we collect and discuss some problems that we have not been able to solve
with the hope that they will stimulate further investigation.

First of all, we define the Fock-Sobolev space. For a fixed non-negative integer $m$, the Fock-Sobolev space $F^{2,m}$ consisting of entire functions $f$ on $\mathbb{C}^n$ such that
$$\sum_{|\alpha|\leq m} \|\partial^\alpha f\|_2<\infty,$$
where $\|\cdot\|_2$ is the norm in $F^2.$  The Fock-Sobolev space is a  Hilbert space of holomorphic functions that is closely related to the Fock space.  However, there exists a fundamental difference in the geometries of the Fock and Fock-Sobolev spaces, see \cite{Qin}.
 Thus the following question is open and quite challenging.
\begin{proA}
 Write
$$\varphi=f+\overline{g}\quad\text{and}\quad\psi=u+\overline{v}$$
where $f,g,u,v\in \mathcal{\varepsilon}(\mathbb{C}^n)\cap H(\mathbb{C}^n)$.
Suppose that $h\in L^2(\mathbb{C}^n,d\mu)$ and $m$ is a positive integer. Determine the $\varphi$ and $\psi$ for which $T_\varphi T_\varphi=T_h$  on the Fock-Sobolev space $F^{2,m}$.
\end{proA}
Another unanswered question we would like to discuss in this section are related commuting Toeplitz operaors on the Fock-Sobolev space.
\begin{proB}
Suppose that $f,g,u,v\in \varepsilon(\mathbb{C})\cap H(\mathbb{C})$ and $m$ is a positive integer. Let $\varphi=f+\overline{g}$ and $\psi=u+\overline{v}$. Determine the $\varphi$ and $\psi$ for which the commutant $[T_\varphi,T_\psi]=0$  on the Fock-Sobolev space $F^{2,m}(\mathbb{C})$.
\end{proB}

\providecommand{\MR}{\relax\ifhmode\unskip\space\fi MR }
\providecommand{\MRhref}[2]{%
  \href{http://www.ams.org/mathscinet-getitem?mr=#1}{#2}
}
\providecommand{\href}[2]{#2}


\begin{thebibliography}1
\bibitem{Ahern}
P. Ahern, On the range of the Berezin transform, J. Funct. Anal., 215 (1) (2004) 206-216.
 \bibitem{Aleman}
A. Aleman,  D. Vukoti\'{c},  Zero products of Toeplitz operators, Duke Math. J., 148 (3) (2009) 373-403.
  \bibitem{Axler}
 S. Axler, \v{Z}. \v{C}u\v{c}ukovi\'{e}, Commuting Toeplitz operators with harmonic symbols, Integral Equ.
Oper. Theory, 14 (1) (1991) 1-12.
\bibitem{Bauer}  Bauer, W., Choe, B. R., Koo, H., Commuting Toeplitz operators with pluriharmonic symbols on the Fock space, J. Funct. Anal., 268 (10) (2015) 3017-3060.
    \bibitem{Bauer1}
    W. Bauer, T. Le, Algebraic properties and the finite rank problem for Toeplitz operators on the
Segal-Bargmann space, J. Funct. Anal., 261 (9) (2011) 2617-2640.
\bibitem{Bauer2}
W. Bauer, Y.J. Lee, Commuting Toeplitz operators on the Segal-Bargmann space, J. Funct. Anal.,
260 (2) (2011) 460-489.
    \bibitem{Brown}
    A. Brown, P. Halmos, Algebraic properties of Toeplitz operators, J. Reine Angew. Math., 213 (1963/1964) 89-102.
     \bibitem{Choe}
    B.R. Choe, H. Koo, Zero products of Toeplitz operators with harmonic symbols, J. Funct. Anal.,
233 (2) (2006) 307-334.
\bibitem{Choe1}
 B.R. Choe, H. Koo, Y.J. Lee, Sums of Toeplitz products with harmonic symbols, Rev. Mat. Iberoam.,
24 (1) (2008) 43-70.
\bibitem{Choe2}
 B.R. Choe, H. Koo, Y.J. Lee, Finite sums of Toeplitz products on the polydisk, Potential Anal.,
31 (3) (2009) 227-255.
\bibitem{Choe3}
 B.R. Choe, H. Koo, Y.J. Lee, Toeplitz products with pluriharmonic symbols on the Hardy space
over the ball, J. Math. Anal. Appl., 381 (1) (2011) 365-38.
\bibitem{Cho}
 H.R. Cho, J.D. Park, K. Zhu, Products of Toeplitz operators on the Fock space. Proc. Amer. Math.
Soc., 142 (7) (2014) 2483-2489.
\bibitem{Ding}
 X. Ding, Y. Qin, D. Zheng, A theorem of Brown-Halmos type on the Bergman space modulo finite
rank operators, J. Funct. Anal., 273 (9) (2017) 2815-2845.
\bibitem{Guo}
 K. Guo, S. Sun, D. Zheng, Finite rank commutators and semicommutators of Toeplitz operators
with harmonic symbols, Ill. J. Math., 51 (2) (2007) 583-596.
\bibitem{Le}
T. Le, A. Tikaradze, A generalization of the Brown-Halmos theorems for
the unit ball, Adv. Math., 404 (2022) 108411.
\bibitem{MA}  P. Ma, F. Yan, D. Zheng, K. Zhu, Products of Hankel operators on the Fock space, J. Funct.
Anal.,  277 (2019) 2644-2663 .
\bibitem{Qin}
 J. Qin, Semi-commuting Toeplitz operator on Fock-Sobolev spaces, Bull. Sci. math., 179 (2022) 103156.
\bibitem{Zheng}
 D. Zheng, Commuting Toeplitz operators with pluriharmonic symbols, Trans. Am. Math. Soc.,
350 (4) (1998) 1595-1618.
\end{thebibliography}
\end{document}